\documentclass[12pt,reqno]{amsart}
\usepackage[nospace,noadjust]{cite}

\usepackage{verbatim,amscd,amssymb}
\usepackage{graphicx}
\usepackage{amsmath}
	\usepackage{amsmath}
\usepackage{amsfonts}
\usepackage{amssymb}
\usepackage{inputenc}

\usepackage{enumitem}
\usepackage{float}

\usepackage[active]{srcltx}

\graphicspath{ {./images/} }

\usepackage{fontenc}

\usepackage{float}
\newtheorem{theorem}{Theorem}[section]
\newtheorem{lemma}[theorem]{Lemma}
\newtheorem{cor}[theorem]{Corollary}

\theoremstyle{definition}
\newtheorem{definition}[theorem]{Definition}

\theoremstyle{remark}
\newtheorem{remark}[theorem]{Remark}

\theoremstyle{assumption}

\numberwithin{equation}{section}

\newcommand{\IET}{\mathrm{IET}}

\newcommand{\card}{\mathrm{Card}}

\newcommand{\sha}{\succ\mkern-14mu_s\;}

\begin{document}
	
	\date{\today}
	
	\title[Forcing Minimal Interval Patterns as Interval Exchange Transformations]{Forcing Minimal Interval Patterns as Interval Exchange Transformations}

	\author{Sourav Bhattacharya}

	\address[Dr. Sourav Bhattacharya]
	{Department of Mathematics, Visvesvaraya National Institute Of Technology Nagpur, 
 Nagpur, Maharashtra 440010, India}
	
	\email{souravbhattacharya@mth.vnit.ac.in}
	
	\subjclass[2010]{Primary 37E05, 37E10, 37E15; Secondary 37E45}
	
	\keywords{Sharkovsky Theorem, rotation numbers, ergodicity, unfolding numbers }

	\begin{abstract}
		We prove that any \emph{over-twist pattern}  is \emph{conjugate} to an \emph{interval exchange transformation}  with \emph{bounded} number of \emph{segments of isometry}, \emph{restricted} on one of its \emph{cycles}.  The \emph{bound} is \emph{independent} of the \emph{period} and \emph{over-rotation number} of the \emph{over-twist pattern} and depends only on its \emph{modality}.  
			
	\end{abstract}

	\maketitle

\section{INTRODUCTION}\label{s:intro}
A classical problem in \emph{one-dimensional dynamical systems} is to classify \emph{cycles} (or \emph{periodic orbits}) of a \emph{continuous} interval map. Any such \emph{classification} relies on \emph{properties} of \emph{cycles} that are \emph{invariant} under \emph{topological conjugacy}. Considering \emph{period} of these \emph{cycles} as the \emph{property} under study, we get an elegant \emph{classification} of \emph{cycles} given by the renowned Sharkovsky Theorem (\cite{shatr}). It elucidates how one \emph{period} of a \emph{cycle} \emph{forces} the \emph{existence} of other \emph{periods}, translating the \emph{problem} of \emph{classification} of \emph{cycles} to one of \emph{forcing among cycles}. However, Sharkovsky Theorem is inapplicable for \emph{classification} if two \emph{cycles} have the same \emph{period}. In such situations, we need to search for other  \emph{criterion}.

Now, one can consider the \emph{finest} possible \emph{criterion}: \emph{group} \emph{cycles} into \emph{equivalence classes},  called \emph{patterns} based upon associated \emph{cyclic permutation} and ascertain a rule of \emph{forcing} among these \emph{groups}. However, the \emph{forcing relation} among \emph{patterns} is a \emph{partial ordering}  (\cite{Ba}); the \emph{results} in the case are much \emph{more} \emph{complicated} than the one obtained for \emph{period} and a coherant description of the structure of the set of \emph{patterns} \emph{exhibited} by \emph{periodic orbits} of a given \emph{continuous interval} map is not possible (see \cite{Ba}). 

This motivated researchers to look for other middle-of-the-road ways of describing \emph{cycles} in between the theories based upon \emph{periods} and \emph{permutations}. One such prominent work is that of Blokh and Misiurewicz (\cite{BM1}). They introduced a new \emph{invariant} of \emph{topological conjugacy} for \emph{cycles}: \emph{over-rotation numbers}, by extending \emph{rotation theory} to \emph{interval maps}.  The \emph{classification rule} induced by \emph{over-rotation numbers}, is on one hand much \emph{finer} than the \emph{period-based classification}; on the other hand, the \emph{forcing relation} on the \emph{collection} of all \emph{cycles} induced by the \emph{theory} is \emph{linear},  like the case of \emph{periods} and \emph{allows} for a transparent description.

In \emph{over-rotation theory} for \emph{interval maps}, one of the \emph{most important} objects of study are,  the \emph{forcing minimal cycles} with a given \emph{over-rotation number}, called \emph{over-twist cycles}. These are the \emph{simplest cycles} with a given \emph{over-rotation number}: a \emph{cycle} $P$ with \emph{over-rotation number} $\rho$ is  \emph{over-twist} if there exists an interval map $f_P$ such that $P$ is the \emph{unique cycle} of $f_{P}$ with \emph{over-rotation number} $\rho$.  A \emph{necessary} and \emph{sufficient condition} for a given \emph{cycle} to be an \emph{over-twist cycle} was established in \cite{BM2}. Further, description of the \emph{dynamics} of all possible \emph{over-twist cycles} of \emph{modality} less than or equal to 2 were obtained in \cite{BS} and \cite{BB2}.

The goal of this \emph{paper} is to present a new perspective of looking at these \emph{over-twist cycles} by connecting them with \emph{interval exchange transformations}. An \emph{interval exchange transformation} (IET) is a \emph{right continuous bijective map} from an \emph{interval} to \emph{itself} such that the \emph{interval} can be \emph{subdivided} into a \emph{finite} number of \emph{pieces}, called \emph{segments of isometry}, on each of which the map acts as a \emph{translation}.   Since Keane’s seminal paper \cite{KM}, the field of IETs has witnessed a remarkable surge in interest. These maps have fascinated researchers not only from a theoretical standpoint, as evidenced by the substantial body of work in this area (see \cite{VM} for a comprehensive review), but also due to their numerous applications. Notably, IETs play a crucial role in the study of \emph{polygonal billiards} (\cite{NA,BKM}) and \emph{surface flows} \cite{EPS}, among others.

Our \emph{results} seeks to offer a unified framework to study problems in \emph{over-rotation theory}, IETs, and the various disciplines related to IETs. The main \emph{results} of this paper are:

\begin{enumerate}
	\item We show that any \emph{over-twist cycle}  of \emph{modality} $m$ is \emph{conjugate} to an IET with \emph{atmost} $(m+1)(m+2)$ \emph{segments of isometry}, \emph{restricted} on one of its \emph{cycles}.  
	
	\item In particular, we show: 
	\begin{enumerate}
		\item any \emph{unimodal} \emph{over-twist cycle}  is \emph{conjugate} to an IET with $3$ \emph{segments of isometry}, \emph{restricted} on one of its \emph{cycles}. 
		
		\item any \emph{bimodal} \emph{over-twist cycle}  is \emph{conjugate} to an IET with $4$ \emph{segments of isometry}, \emph{restricted} on one of its \emph{cycles}. 
	\end{enumerate}

\end{enumerate}

This \emph{paper} contains 4 \emph{sections}:

\begin{enumerate}
	\item \emph{Section 1} briefly \emph{outlines} the \emph{paper}.
	
\item \emph{In Section 2}, we present all essential \emph{definitions}, \emph{lemmas}, and \emph{theorems} necessary to describe the results.

\item In Section 3 we \emph{prove}  our \emph{results}.

\end{enumerate}

\section{PRELIMINARIES }\label{s:prelim}

\subsection{Modality}\label{ss:modality}

Let $U$ and $V$ be given \emph{topological spaces} and $f: U \to V$ be a given \emph{function}. If there exists a subset $S$ of $U$ such that the \emph{inverse image} of every element of $f(S)$ is a \emph{connected} subset of $U$, we say that $f$ is \emph{monotone} on $S$. Further, if $S$ is the \emph{maximal open subset} of $U$ satisfying this property, we say that $S$ is a \emph{region of monotonicty} of $f$. If $U,V,S \subseteq \mathbb{R}$,  $f$ is \emph{monotone} on $S$, means that, for all, $x,y \in S$, $x \leqslant y$, either $f(x) \leqslant f(y)$, or $f(x) \geqslant f(y)$. In the \emph{first case}, we say $f$ is \emph{order-preserving} or \emph{monotonically increasing} on $S$; in the \emph{second case}, we say $f$ is \emph{order-reversing} or \emph{monotonically decreasing} on $S$.

	We call a \emph{map} $f: [0,1] \to [0,1]$, \emph{piece-wise-monotone} if $[0,1]$ can be \emph{subdivided} into \emph{finite number} of \emph{sub-intervals}, such that $f$ is \emph{monotone} on each of these \emph{sub-intervals}. We call the smallest number of such \emph{sub-intervals}, the \emph{lap count} of $f$. If the \emph{lap count} is $2$, we call $f$, \emph{unimodal}. If the \emph{lap count} is $3$, we call $f$, \emph{bimodal}. In this paper, we will consider only \emph{piece-wise monotone maps}. 

\subsection{Sharkovsky Theorem}\label{ss:sharkov}

Consider the   \emph{order} $\sha$ among all \emph{natural numbers} called the \emph{Sharkovsky Ordering} defined as follows: 

$$3\sha 5\sha 7\sha\dots\sha 2\cdot3\sha 2\cdot5\sha 2\cdot7 \sha \dots $$
$$
\sha\dots 2^2\cdot3\sha 2^2\cdot5\sha 2^2\cdot7\sha\dots\sha 8\sha
4\sha 2\sha 1$$

The Shakovsky Theorem is as follows: 
\begin{theorem}[\cite{shatr}]\label{sharkov}

Suppose $m,n \in \mathbb{N}$ with $m \sha n$. Then, any continuous map $f:[0,1] \to [0,1]$ having a cycle of period $m$ must also have a cycle of period $n$. 

\end{theorem}

\subsection{Patterns}\label{ss:patterns}

Let $\mathcal{C}$ be the \emph{collection} of all \emph{cycles}. We define an \emph{equivalence relation} $\sim$ on $\mathcal{C}$ as follows: for $P,Q \in \mathcal{C}$, we \emph{write} $P \sim Q$, if $P$ and $Q$ have the same associated \emph{cyclic permutation}. Each \emph{equivalence class} under $\sim$ is called a \emph{pattern}. A \emph{cycle} $P$ is said to \emph{exhibit} a \emph{pattern} $\pi$ if $P$ is contained in the \emph{equivalence class} \emph{corresponding} to $\pi$. 

Given two \emph{patterns} $\pi$ and $\mu$, we say that $\pi$ \emph{forces} $\mu$, and \emph{write} $\pi \gg \mu$, if any \emph{interval map} $f$ having a \emph{cycle} which \emph{exhibits} \emph{pattern} $\pi$, must have a \emph{cycle} which \emph{exhibits} $\mu$.

\begin{theorem}[\cite{Ba}]\label{partial}
The ordering induced by the forcing relation $\gg$ on the collection of all patterns is a partial ordering. 
\end{theorem}

Consider a \emph{cycle} $P$ which exhibits a \emph{pattern} $\pi$. Let the \emph{left-most} and \emph{right-most points} of $P$ be $a$ and $b$ respectively.  We call each \emph{component} of $P -[a,b]$, a $P$-\emph{basic interval}. The \emph{map} $f: [a,b] \to [a,b]$ obtained by \emph{extending} the \emph{map associated} with the \emph{cycle} $P$ from $P$ to $[a,b]$ is called a $P$-\emph{linear map}.

\begin{theorem}[\cite{Ba}]\label{frocing:patterns}
The necessary and sufficient condition for a pattern $\nu$ to be forced by a pattern $\mu$ is that the $P$-linear map $f$ where $P$ exhibits $\mu$ contains a cycle $Q$ which exhibits pattern $\nu$. 
\end{theorem}

\subsection{Over-rotation numbers}\label{ss:over:number}

Consider a \emph{continuous} map $f:[0,1] \to [0,1]$. We assume $f$ has a \emph{cycle} of \emph{period} strictly greater than $1$; we make this assumption because otherwise for any $x \in [0,1]$, the \emph{sequence} $\{ f^n(x) : n \in \mathbb{N}\}$ \emph{converges} to some \emph{fixed point} of $f$ and hence the dynamics is not interesting. Given a \emph{cycle} $P$ of $f$ of \emph{period} $q$ strictly greater than $1$, we \emph{count} the \emph{number} of \emph{points}:  $x \in P$ such that $f(x) - x$ and $f^2(x) - f(x)$ have different \emph{signs}. If there are $m$ such \emph{points}, the \emph{pair} $(\frac{m}{2}, q)$ and the \emph{number} $\frac{m}{2q}$ are called \emph{over-rotation pair} and \emph{over-rotation number} of the \emph{cycle} $P$ and denoted by $orp(P)$ and $orn(P)$ respectively.

In an \emph{over-rotation pair} $(s,t)$,  $s$ and $t$ are \emph{integers} and $0 <  \frac{s}{t} \leqslant \frac{1}{2}$. Suppose $f$ has a unique \emph{fixed point} $a$, $f(x) > x$ if $x< a$ and $f(x) < x$ if $x>a$. In this case, $s$ in the \emph{over-rotation pair} $(s,t)$ is equal to the \emph{number} of \emph{points} which are \emph{located} to the \emph{right} of the \emph{fixed point} $a$ but mapped to the \emph{left} of $a$.

\begin{theorem}[\cite{BM1}]\label{overrotation:pairs}
	Suppose $p,q,r,s$ are integers which satisfies one of the following two conditions: 
	
	\begin{enumerate}
		
		\item  $\frac{p}{q}< \frac{r}{s}$.
		
		\item $\frac{p}{q}=\frac{r}{s} = \frac{k}{\ell}$ where $k$ and $\ell$ are coprime integers and $\frac{p}{k} \sha \frac{r}{k}$ where $\sha$ denotes Sharkovsky ordering. 
	\end{enumerate}
	
	Then, any continuous interval map $f$ having a cycle of over-rotation pair $(p,q)$ must have a cycle of over-rotation pair $(r,s)$. 
\end{theorem}

By Theorem \ref{sharkov}, any \emph{continuous interval map} having a \emph{cycle} of \emph{period} \emph{strictly greater} than $1$ must have a \emph{cycle} of \emph{period} $2$. Thus, the following result thus follows from Theorem \ref{overrotation:pairs}:

\begin{cor}\label{ovr:int}
	The closure of the set of over-rotation numbers of periodic points of a continuous interval map $f:[0,1] \to [0,1]$ forms an interval of the form $[r_f, \frac{1}{2}]$ where $ 0 \leqslant r_f < \frac{1}{2}$, called the over-rotation interval of $f$. 
\end{cor}

By definition, all \emph{cycles} which exhibit the same \emph{pattern} have the same \emph{over-rotation number} and same \emph{over-rotation pair}. We define, the \emph{over-rotation number} and \emph{over-rotation pair} of a \emph{pattern} $\pi$, to be equal to the \emph{over-rotation number} of any \emph{cycle} which \emph{exhibits} $\pi$. If $P$ is a \emph{cycle} which exhibits $\pi$, we call the \emph{over-rotation interval} of the $P$-\emph{linear} map $f_P$, the \emph{over-rotation interval} \emph{forced} by $\pi$. 

\subsection{Over-twist patterns}\label{ss:over:patt}

We call a \emph{pattern} $\pi$ \emph{over-twist} if it doesn't \emph{force} other \emph{patterns} with the same \emph{over-rotation number}. By Theorem \ref{overrotation:pairs}, it follows that such a \emph{pattern} cannot \emph{force} \emph{patterns} with \emph{over-rotation number} less than that of the \emph{pattern}; thus for a given \emph{over-rotation number}, \emph{over-twist patterns} are \emph{forcing minimal}. Theorem \ref{overrotation:pairs} also tells us that of $\pi$ is an \emph{over-twist pattern} with \emph{over-rotation pair} $(p,q)$, then $p$ and $q$ are \emph{coprime}. The following theorem tells us about the existence of \emph{over-twist patterns}:

\begin{theorem}[\cite{B1}, \cite{blo95a}, \cite{B2}]\label{t:sharp}
	For any piece-wise monotone continuous function $f$ having over-rotation interval $I_f$ and for any  $ \rho \in I_f \cap \mathbb{Q}$, $f$ has a cycle which exhibits an over-twist pattern with over-rotation number $\rho$. 
\end{theorem}

A \emph{pattern} $\pi$ is called \emph{convergent} if for any \emph{cycle} $P$ which \emph{exhibits} $\pi$, the $P$-\emph{linear} map $f$ has a unique \emph{fixed point}. The following result follows:

\begin{theorem}[\cite{BM1}]\label{div:res}
	If  a pattern $\pi$ is not convergent, then for any arbitrary over-rotation number $(p,q)$, $\pi$ forces a pattern $\nu$ with over-rotation pair $(p,q)$. 
\end{theorem} 

\begin{cor}\label{conv}
	If $\pi$ is an over-twist pattern, then $\pi$ is convergent. 
\end{cor}

From Corollary \ref{conv}, it follows that to study \emph{over-twist patterns}, it is sufficient to consider  maps with a unique fixed point. Let $\mathcal{U}$ be the \emph{collection} of all \emph{piece-wise monotone continuous maps} with a unique fixed point. We will always denote such a fixed point by $a$. 

\subsection{Green patterns}\label{ss:green}

Consider a \emph{cycle} $P$ of a \emph{map} $ f \in \mathcal{U}$. A point $g$ of $P$ is called \emph{green} if it \emph{maps} to the \emph{same side} of $a$ as $g$. A point $b$ of $P$ is called \emph{black} if it maps to the \emph{side} of $a$ not containing $b$. Let us denote by $\mathcal{G}(P)$ and $\mathcal{B}(P)$, the \emph{collection} of all \emph{green points} and \emph{black points} of $P$ respectively. We call a \emph{cycle} $P$, \emph{green} if $f|_{\mathcal{G}(P)}$ is \emph{increasing} and $f|_{\mathcal{B}(P)}$ is \emph{deceasing}. We call a \emph{pattern} $\pi$ \emph{green},  if any \emph{cycle} which \emph{exhibits} $\pi$ is \emph{green} (See Figure \ref{drawing1}). 

\begin{theorem}\label{nec:green}
	A necessary condition for a pattern $\pi$ to be over-twist is that $\pi$ is green. 
\end{theorem}

\begin{figure}[H]
	\caption{A \emph{green pattern} $\pi$}
	\centering
	\includegraphics[width=0.7 \textwidth]{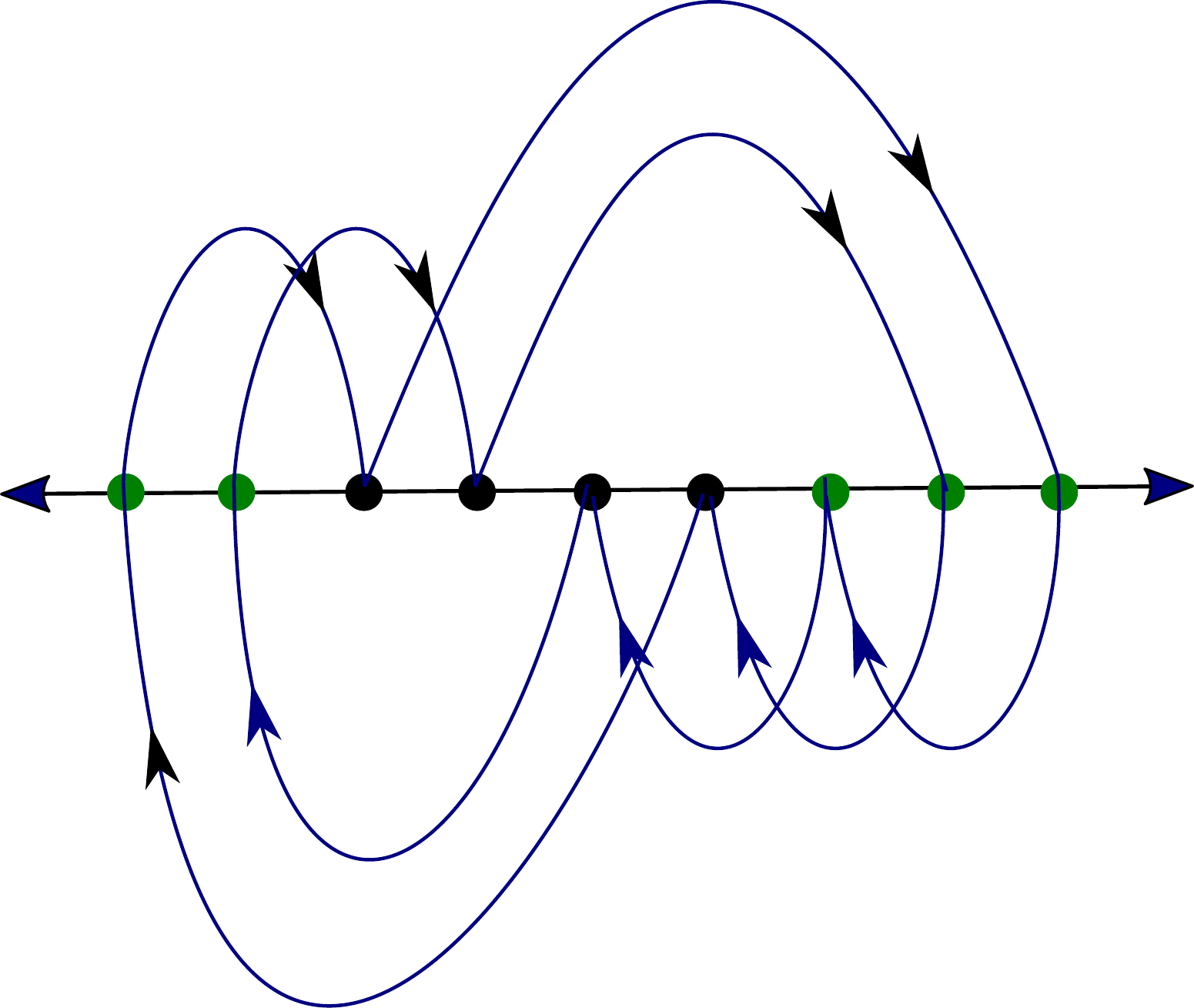}
	\label{drawing1}
\end{figure}
\subsection{Loops of intervals}\label{ss:loops}

Let $f$ be a \emph{continuous} interval map. An \emph{interval} $L$ is said to $f$-\emph{cover} an \emph{interval} $K$ if $f(L) \supseteq K$.  By an $f$-\emph{loop of intervals} $\alpha$ of \emph{length} $k$, we shall mean a \emph{finite} a  \emph{sequence} of $k$  \emph{intervals}  $\{I_i:  i=1,2,\dots k\}$ such that  $I_j$  $f$-\emph{covers} $I_{j+1}$ for all $ j =1,2, \dots k$ and also $I_{k}$ $f$-\emph{covers} $I_{1}$. We denote such  an $f$-\emph{loop of intervals} by $\alpha : I_1 \to I_2 \to \dots I_{k-1} \to I_k \to I_1$.

The following result follows: 

\begin{lemma}[\cite{alm00}]\label{ALM2}
	Given any $f$-loop of intervals  $\alpha : I_1 \to I_2 \to $ $ \dots I_{k-1} $ $  \to I_k \to I_1$ of length $k$, there exists a periodic point $x \in I_1$ such that $f^j(x) \in I_j $ for $
	 j =1,2, \dots k$ and $ f^k(x)=x$.  The cycle $Q = \{ x, $ $ f(x), $ $ f^2(x), \dots $ $  f^{k-1}(x) \}$ is called  the cycle associated with the loop $\alpha$. 
	
\end{lemma}

Let $f \in \mathcal{U}$. An $f$-\emph{loop} of \emph{intervals} $\alpha : I_1 \to I_2 \to \dots I_{k-1} \to I_k \to I_1$ is called an $f$-\emph{admissible loop of intervals} if \emph{one} \emph{end point} of each of \emph{intervals} $I_j$, $j=1,2,\dots k$ is the \emph{fixed point} $a$. Given a \emph{cycle} $P$ of \emph{period} $q$, the \emph{admissible loop of intervals} $\alpha : [x,a] \to [f(x), a] \dots [f^{q-1}(x), a] \to [x,a]$ for any $x \in P$, is called the \emph{fundamental admissible loop of intervals associated} with the \emph{cycle} $P$.

\subsection{Interval Exchange Transformation}\label{ss:IET}

	A \emph{bijection}  $T$ from $[0,1]$ to itself is called an $(n, k)$-\emph{Interval Exchange Transformation} on $[0,1]$, denoted by $(n, k)-\IET$, if: 
	
		\begin{enumerate}
		\item there exists some \emph{partition} of $[0,1]$:  $\mathcal{P}_{n} = \{0= x_0, x_1, \dots, x_{n}=1\}$ such that the \emph{restriction} of $T$ to each \emph{sub-interval}  $I_j = [x_i, x_{i+1})$, $j=0,1,2, \dots n-1$ is an \emph{affine} map having slope $\pm 1$.

		\item there are \emph{precisely} $k$ \emph{sub-intervals} $I_{j}$ such that $T|_{I_j}$ has \emph{slope} $-1$. 
	\end{enumerate}
	
	We call the  \emph{sub-intervals} $I_{j}, j=0,1,2, \dots n-1$,  \emph{intervals of isometry} and the points $x_j$,  \emph{separation points}. An \emph{interval of isometry} $I_j$ is said to be \emph{oriented} or \emph{flipped} according to whether $T|_{I_j}$ has \emph{slope} $+1$ or $-1$. 
	
\begin{figure}[H]
	\caption{A $(4,2)$-IET with \emph{segments} of \emph{isometry} $I_0$, $I_1$, $I_2$ and $I_3$}
	\centering
	\includegraphics[width=0.8 \textwidth]{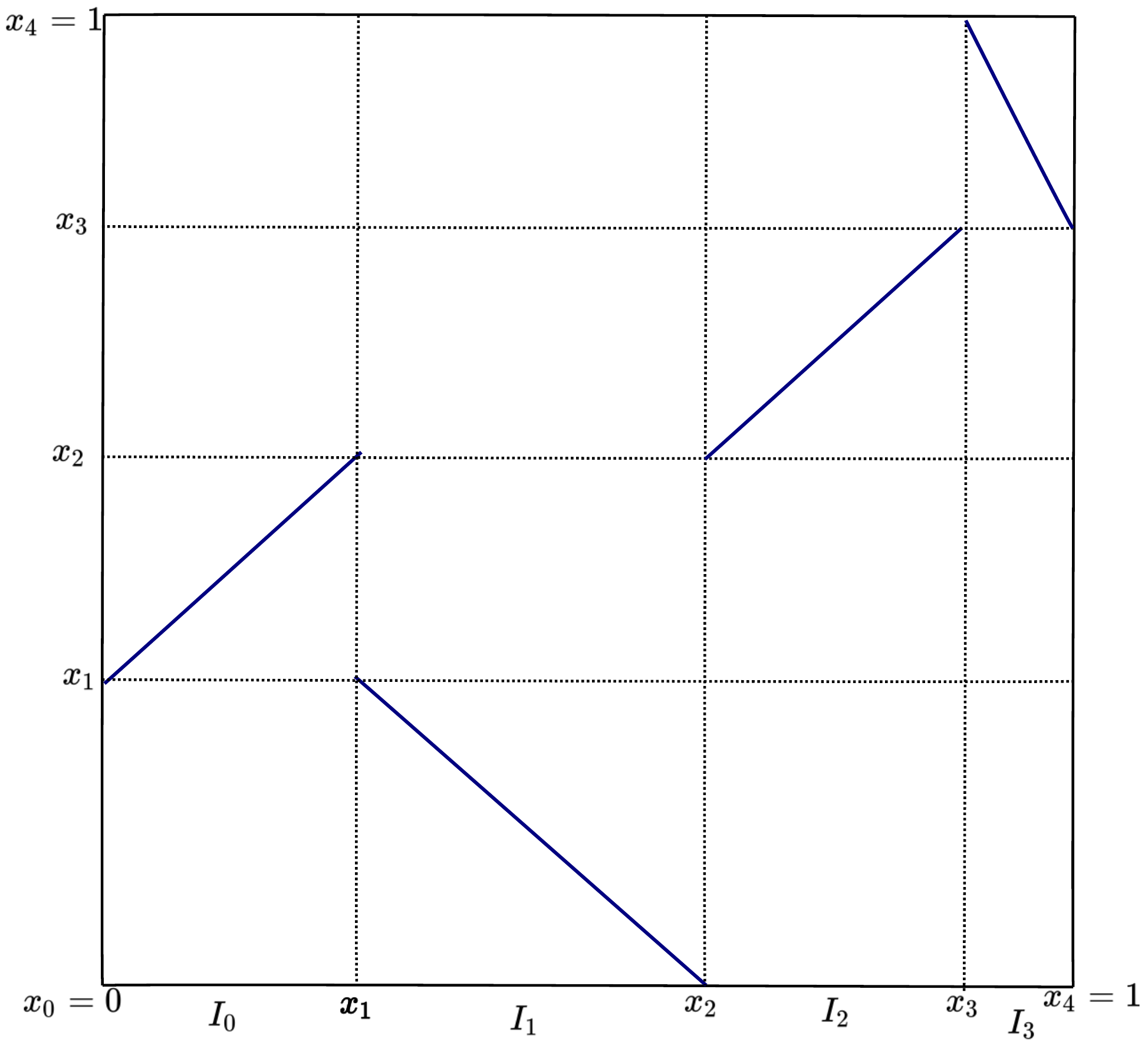}
	\label{drawing2}
\end{figure}

By a \emph{double permutation} on $n$ \emph{symbols}, we mean an \emph{injective function} $\pi_n^d : \{ 1,2 \dots n \} \to \{ -1,1,-2,2,-3,3,-4,4, \dots , -n,n\}$ such that $|\pi_n^d| :  \{ 1,2 \dots n \} \to \{ 1,2 \dots n \} $ is a \emph{permutation} on $n$ \emph{symbols}. Let $\mathcal{D}_n$ be the \emph{collection} of all \emph{double permutation} on $n$ \emph{symbols}. We can parameterize the \emph{set} of all $(n,k)$-IET on $[0,1]$ as follows. Let $\mathbb{R}_{+} = (0, \infty)$ and $\mathcal{C}_n = \mathbb{R}^n_{+} \times \mathcal{D}_n$. Then, we associate an $(n,k)$-IET with the \emph{ordered pair} $(\vec{\lambda}, \pi_n^d) \in \mathcal{C}_n$ where:

\begin{enumerate}
	\item $\vec{\lambda} = (\lambda_0, \lambda_1, \dots \lambda_{n-1})$ where $\lambda_{j} = x_{j+1} - x_j$ is the \emph{length} of the $i$-th \emph{interval} $I_j$, $j = 0,1,2, \dots n-1$.
	
		\item $\pi^{d}_n\in \mathcal{D}_n$ is defined as follows:
	
	\begin{enumerate}
		\item for any $i=1,2, \dots n$, $|\pi^{d}_n(i)|$ gives \emph{location} of $T(I_{i-1})$ in the \emph{set} $\{ T(I_{i-1}) : i=1,2, \dots n\}$, in the \emph{usual order} in $\mathbb{R}$. 
		\item 	$\pi^{d}_n(i) >0$ if $T|_{I_{i-1}}$ has \emph{slope} $1$ and $\pi^{d}_n(i) <0$ if $T|_{I_{i-1}}$ has \emph{slope} $-1$, for $i=1,2, \dots n$.

	\end{enumerate} 
	
\end{enumerate}

The following \emph{result} follows easily: 
\begin{lemma}
	$T|_{I_i}$ is a translation by $t_i \in \mathbb{R}_{+}$ where $t_i $ $=  \displaystyle \sum_{|\pi^{\sigma}_n(j)| <|\pi^{\sigma}_n(i)|} \lambda_j  -  \displaystyle  \sum_{j <i} \lambda_j $. 
\end{lemma}

\section{Over-twist patterns as an Interval Exchange Transformation (IET) }\label{bimodal:section}

For  sets $C$ and $D$, we write $C \leqslant D$ if $c \leqslant  d$ for all $c \in C$ and $d \in D$. For any set $A$, let $\card(A)$, denote its \emph{cardinality}.

\begin{definition}\label{without:expansion}
	For a \emph{finite set} $A \subset \mathbb{R}$, and an \emph{interval map} $f$, we say that $f|_{A}$ \emph{maps} without \emph{expansion} if any two \emph{consecutive points} of $A$ are \emph{mapped} onto \emph{consecutive points} of $f(A)$. 
\end{definition}

\begin{theorem}\label{P:linear:IET}
	Suppose  a cycle $P$ of a $P$-linear map $f: [0,1] \to [0,1]$ can partitioned into $n$ subsets:  $K_1, K_2, \dots K_{n}$ $(K_1<  K_2 <  K_3 <  \dots K_n)$, called blocks with the following properties:
	
	\begin{enumerate}
	\item $f|_{K_i}$ is monotone and maps without expansion for $i=1,2,\dots n$, 
	
	\item there are precisely $k$ indices $j$ such that $f|_{K_j}$ is decreasing. 
	\end{enumerate}

 Then, there exists a $(n,k)$- IET, $T:[0,1] \to [0,1]$ and a cycle $Q$ of $T$ such that $f|_P$ is conjugate to $T|_{Q}$. 
\end{theorem}

\begin{proof}
	Let the \emph{elements} of $P$ be $0=x_1, x_2, \dots x_q=1$. Define, a \emph{function} $\psi$ on $[0,1] $ as follows: define, $\psi$ on $P$  by $ \psi(x_i)  =  y_i$ where $y_i = \frac{2i-1}{2q}$ for $i=1,2 \dots q$; \emph{extend} $\psi$ from $P$ to $[0,1]$ by \emph{defining} it \emph{linearly} on each \emph{component} of $[0,1] -P$.

	We \emph{define} $T:[0,1] \to [0,1]$ as follows: 
	\begin{enumerate}
		\item for each $y_i$, $i=1,2,\dots q$, \emph{define} $T(y_i) = \psi(f(x_i))$, 
		
		\item for each $y_i$, $i=1,2,\dots q$, the \emph{graph} of $T$ in the \emph{interval} $\mathcal{N}(y_i) = (y_i - \frac{1}{2}, y_i + \frac{1}{2})$ is the \emph{straight  line segment} $L_i$ of \emph{slope} $\pm 1$ \emph{passing} through the point $(y_i, T(y_i))$;   \emph{slope} of $L_i$ is $-1$ if  $x_i $ lies in the \emph{block} $K_i$ and $f|_{K_i}$ is \emph{decreasing} and \emph{slope} of $L_i$ is $+1$  if $f|_{K_i}$ is \emph{increasing}. 
	\end{enumerate}
	
	Observe that \emph{union} of the intervals $\mathcal{N}(y_i)$, $i=1,2,\dots q$ \emph{yields} the whole $[0,1]$ and hence $T$ is \emph{defined}. Also, since, $f|_{K_i}$ is \emph{monotone} and \emph{maps} without \emph{expansion}, it can be seen easily that if $x_i$ and $x_j$ lie in the same \emph{block} $K_{\ell}$ for some $\ell \in \{1,2, \dots n\}$, then the \emph{line segments} $L_i$ and $L_j$ lie on a  \emph{straight line}. The \emph{construction} yields an $(n,k)$-IET, $T:[0,1] \to [0,1]$.

	Let $Q = \{ y_1, y_2, \dots y_q\}$. Then, by construction, $Q$ is a \emph{periodic orbit} of $T$ and  $\psi$ \emph{conjugates} $f|_P$ with $T|_{Q}$. Hence the result follows. 
\end{proof}

For $x,y \in [0,1]$, let the \emph{notation} $x >_a y $ \emph{denote} that  $x$ and $y$ belong to the \emph{same side} of $a$ and $x$ lies \emph{farther away} from the \emph{fixed point} $a$ \emph{than} $y$. Also, for a \emph{set} $A \subset [0,1]$, let $[A]$ \emph{denote} its \emph{convex hull}. 

\begin{definition}\label{sibling:defn}
	Let $f \in \mathcal{U}$ be of any \emph{modality} $m$. Two \emph{points} $x,y \in [0,1]$ are called \emph{sibling} of each other \emph{under} $f$, if they \emph{map} \emph{under} $f$ to the same \emph{point} of $[0,1]$. In particular, every \emph{point} of $[0,1]$ is its own \emph{sibling} \emph{under} $f$. 
\end{definition}

\begin{definition}\label{c:f}
	
	Let $f \in \mathcal{U}$ be of any \emph{modality} $m>0$. Since, the \emph{collection} of all \emph{siblings} of a \emph{point} $x \in [0,1]$ forms a \emph{closed set}; for every \emph{point} $x \in [0,1]$, we can choose, the \emph{sibling} $s(x)$ of $x$ \emph{closest} to $a$ on the \emph{same  side} of $a$ \emph{containing} $x$. The set $ S(f) = \{ s(x) : x \in [0,1]  \}  $ is called the \emph{special set} of $f$. The \emph{sets} $S(f) \cap [0,a]$ and $S(f) \cap [a,1]$ are called \emph{left special set} and \emph{right special set} of $f$ and denoted by  $S_1(f)$ and $S_2(f)$ respectively. 
	
\end{definition}

 It is easy to see that $f(S(f)) = [0,1]$. Also $S(f)$ is the \emph{union} of \emph{disjoint} \emph{components}; we call each such \emph{component} a \emph{concordant piece} of $f$. If $A$ is a \emph{concordant piece} of $f$ and $A \subseteq S_1(f)$, then we call $A$, a \emph{left concordant piece} of $f$. Similarly, if $A \subseteq S_2(f)$, then we call $A$, a \emph{right concordant piece} of $f$ (See Figure \ref{drawing3}).

 \begin{figure}[H]
 	\caption{A map $f$ with \emph{modality} $m=6$ having \emph{concordant pieces} $K_1, K_2, K_2, K_4$ and $K_5$}
 	\centering
 	\includegraphics[width=0.9 \textwidth]{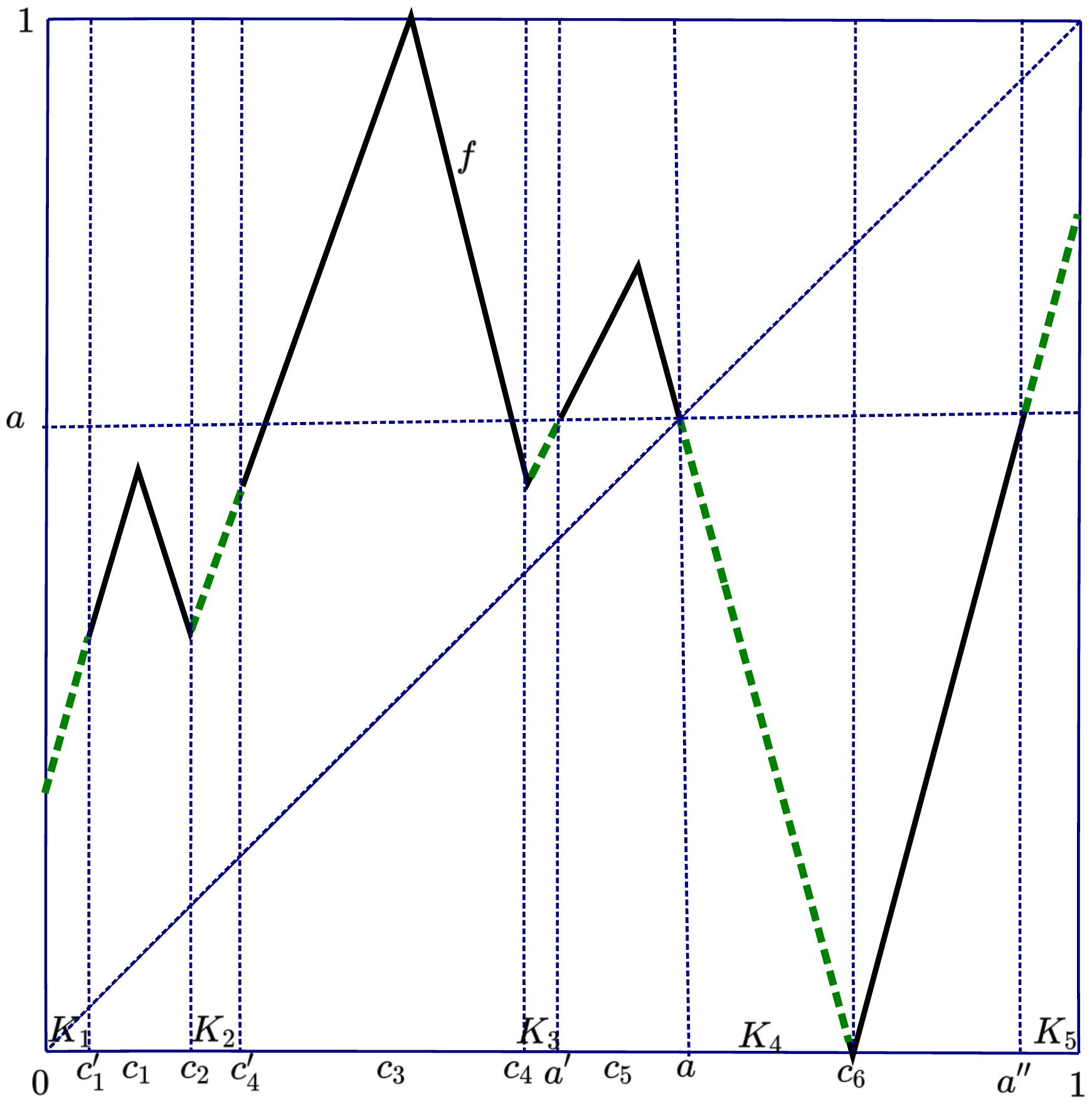}
 	\label{drawing3}
 \end{figure} 

\begin{theorem}\label{first}
	Let $f \in \mathcal{U}$ be of any modality $m>0$. Let $P$ be an over-twist cycle of $f$. Then, $ P \subseteq S(f)$. 
\end{theorem}

\begin{proof}
	By way of contradiction, suppose, there exists $ x \in P$, such that $ x \notin S(f)$. Then, $x$ has a \emph{sibling} $y \neq x$ which is \emph{closer} to the \emph{fixed point} $a$ than $x$. This means, we can \emph{replace} the interval $[x,a]$ in the \emph{fundamental admissible loop} $\alpha$ \emph{associated} with $P$   by $[y, a]$ to form a \emph{new} $f$-\emph{admissible loop} $\beta$.
	
	Let $Q$ be the \emph{cycle} associated with $\beta$ given by Theorem \ref{ALM2}.  By construction, the \emph{over-rotation numbers} of the \emph{cycles}: $P$and $Q$ is the same. We \emph{claim}:  $P \neq Q$. By way of \emph{contradiction}, let us assume otherwise. Let $z$ be the \emph{point} of $P$ \emph{closest} to $a$ on the \emph{side} of $a$ not containing $\alpha$. Then, the interval $[z,a]$ occurs in both the loops $\alpha$ and $\beta$. Choose $k \in \mathbb{N}$, $k< q$ such that $f^k(z) = x$. Since, we have assumed $P =Q$, so $P$ is \emph{associated} with $\beta$; so by \emph{definition} of \emph{loop}, \emph{starting} at $z$ in $[z, a]$ belonging to the \emph{loop} $\beta$ and \emph{walking} along the \emph{loop} $\beta$ for $k$ \emph{steps}, we should \emph{land} at $x$. But, by construction of the \emph{loop} $\beta$, by \emph{walking} $k$ \emph{steps} from the interval $[z, a]$ along $\Delta$, we reach the interval $[y, a]$ and $x >_a y$. So, we arrive at a contradiction. So, $P \neq Q$. 
	
	Thus, $P$ \emph{forces} a \emph{cycle} $Q \neq P$ with the same \emph{over-rotation number} as $P$ contradicting the fact that $P$ is \emph{over-twist}. Hence, the result follows.

\end{proof}

\begin{theorem}\label{second}
	Let $f \in \mathcal{U}$ be of any modality $m$. Then, the number of concordant pieces of $f$ is atmost $m + 4$. 
\end{theorem}

\begin{proof}	
	To prove the result let us introduce a few notations. Let $c_1, c_2, \dots c_{k}$ be the \emph{critical points} of $f$ in $[0,a]$  and $d_1, d_2, \dots d_{\ell}$ be the \emph{critical points} of $f$ in $[a,1]$, $\ell = m-k$.  Let for each $ i \in {1,2, \dots k}$, $c_i'$ be the \emph{sibling} of $c_i$ \emph{nearest} to $c_i$, to the \emph{left} of $c_i$. Similarly, for each $ j \in {1,2, \dots \ell}$, $d_j'$ be the \emph{sibling} of $d_j$ \emph{nearest} to $d_j$, to the \emph{right} of $d_j$. Also, let $a'$ be the \emph{pre-image} of $a$ to the \emph{left} of $a$ \emph{nearest} to $a$ and let $a''$ be the \emph{pre-image} of $a$ to the \emph{right} of $a$ \emph{nearest} to $a$. Define, $\mathcal{I}_1 = \{ c_i : i=1,2, \dots k \}  \cup \{  0\}$, $ \mathcal{I}_2 = \{ c_i': i=1,2, \dots k \}  \cup \{ a,  a' \} $, $\mathcal{I}_3 = \{ d_i' : i=1,2, \dots $ $ \ell  \}  \cup \{  a, a'' \} $ and $\mathcal{I}_4 = \{ d_i: i=1,2, \dots \ell \} \cup \{ 1 \}$.

 Now, let $[s,t]$ be a \emph{left concordant piece} of $f$. Then, it is easy to see that $s \in \mathcal{I}_1$ and $t \in \mathcal{I}_2$ (See Figure \ref{drawing3}). Similarly, if $[u,v]$ be a \emph{right concordant piece} of $f$, then,  $u \in \mathcal{I}_3$ and $v \in \mathcal{I}_4$. Now, \emph{concordant pieces} by definition are disjoint,  each \emph{left concordant piece} can utilize a \emph{distinct} element of $\mathcal{I}_1$ and each  \emph{right concordant piece} can use a \emph{distinct} element of $\mathcal{I}_4$, it follows that, number of \emph{concordant pieces} of $f$ is \emph{atmost}  $(k+1) + (\ell + 1) = k + \ell + 2 = m+2$.  
\end{proof}

\begin{definition}\label{i:green}
	
	Let $f \in \mathcal{U}$ be of any \emph{modality} $m$. Let $P$ be a \emph{periodic orbit} of $f$. 
	
	\begin{enumerate}

		\item A \emph{point} $x \in P$ is called \emph{i-green} if $x$ is the \emph{image} of a \emph{green point}.

		\item A \emph{point} $x \in P$ is called \emph{i-black} if $x$ is the \emph{image} of a \emph{black point}. 
		
	\end{enumerate}
	
\end{definition}

Observe that each \emph{point} of $P$ is either \emph{i-green} or \emph{i-black}. $P$ is \emph{union} of \emph{set} of \emph{consecutive} \emph{i-green} and \emph{i-black points} called \emph{i-green} and \emph{i-black islands} respectively. We shall use the name \emph{i-island} to refer to \emph{both} an \emph{i-black} and an \emph{i-green} island.

\begin{lemma}\label{first:island}
	Let $f \in \mathcal{U}$ be of any modality $m$. Let $P$ be an over-twist periodic orbit of $f$. Then, the first point of $P$ (in the spatial labeling) and the last point of $P$ (in the spatial labeling)  must be $i$-black. 
\end{lemma}

\begin{proof}
Since $f \in \mathcal{U}$, the \emph{points} of $P$ to the \emph{left} of $a$ are \emph{mapped} to the \emph{right} of \emph{itself} and the points of $P$ to the \emph{right} of $a$ are \emph{mapped} to the \emph{left} of \emph{itself} and hence the result follows. 

\end{proof}

\begin{theorem}\label{fifth}
	Let $f \in \mathcal{U}$ be of any modality $m$. Let $P$ be an over-twist periodic orbit of $f$. Then, the sum of the number of i-green island and i-black islands in each side of $a$ is atmost $m+1$. 
\end{theorem}

\begin{proof}
	Let $n_b$ and $n_g$ be the \emph{number} of \emph{i-black} and \emph{i-green} islands in $[0,a]$. By Lemma \ref{first:island},  the \emph{first i-island} in $[0,a]$ must be \emph{i-black}. Let the \emph{i-islands} in $[0,a]$  be ordered as: $B_1 < G_1 < B_2 < G_2 < B_3 < G_3 < \dots < a$ with $B_i$,  $i=1,2, \dots n_b$  and $G_j$, $j=1,2,\dots n_g$ \emph{representing} \emph{i-black} and \emph{i-green islands} respectively. Choose $b_i \in B_i$, $i=1,2, \dots n_b$ and $g_j \in G_j$ for $j=1,2,\dots n_g$. Then, by Theorem \ref{nec:green}, $f^{-1}(g_1) < f^{-1}(g_2) < \dots  < b_1 < g_1 < b_2 < g_2 < \dots  < a  < \dots < f^{-1}(b_2)< f^{-1}(b_1)$. 
	
	It is easy to see that $f^{-1}([b_i, g_i])$ is \emph{not connected}  for all $i=1,2, \dots n_b$. Indeed, $f^{-1}([b_i, g_i])$ has at-least $2$ \emph{components} one to the \emph{left} of $a$ and the other to the \emph{right} if $a$. So, $f|_{[b_i, g_i]}$ is \emph{not monotone} for all $i=1,2, \dots n_b$ and there exists a \emph{critical point} of $f$ in $[b_i, g_{i}]$ for all $i=1,2, \dots n_b$. Similarly,  there exists a \emph{critical point} of $f$ in $[g_i, b_{i+1}]$ for all $i=1,2, \dots n_g$.  It follows that $f|_{[0,a]}$ has \emph{at-least} $n_b+ n_g -1$ \emph{critical points}. So, $n_b+ n_g -1 < m$ and hence $n_b + n_g < m+1$. The \emph{case} for the \emph{interval} $[a,1]$ is similar.

\end{proof}

\begin{figure}[H]
	\caption{\emph{Dynamics} of $i$-\emph{green} \emph{islands}: $G_1$ and $G_2$ and $i$-\emph{black} \emph{islands} $B_1$ and $B_2$}
	\centering
	\includegraphics[width=0.9 \textwidth]{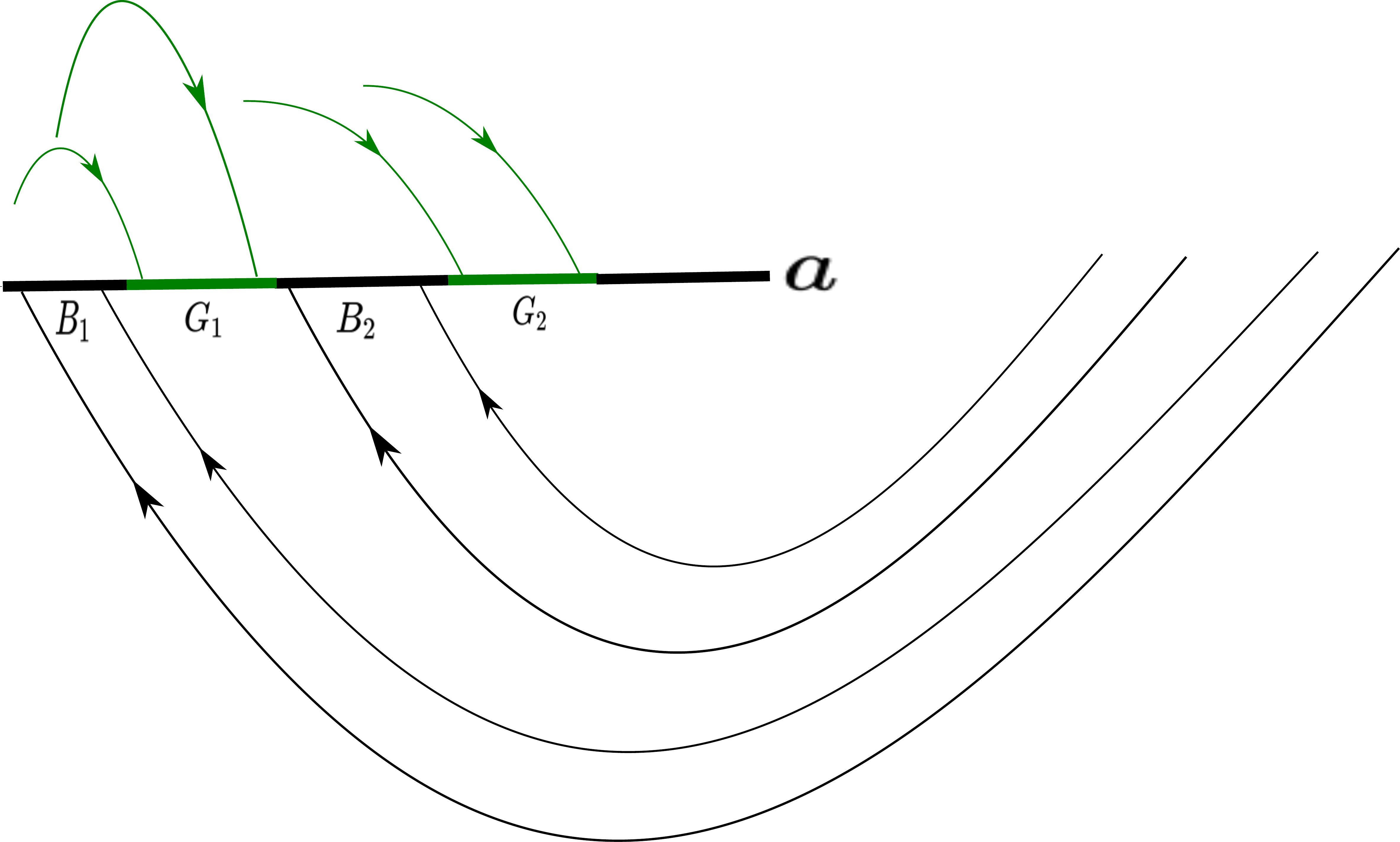}
	\label{drawing4}
\end{figure} 

\begin{theorem}\label{bound:final}
	Let $P$ be  an over-twist periodic orbit of modality $m>0$ and suppose $f\in \mathcal{U}$ be a $P$-linear map. Then, $f|_P$ is conjugate to an $n-IET$ where $n \leqslant (m+2)(m+1)$ restricted on one of its cycles. 
\end{theorem}

\begin{proof}
	
Let $K_1, K_2, \dots K_p$ be the \emph{concordant pieces} of $f$. By Theorem \ref{first}, $P \subseteq \displaystyle \bigcup_{i=1}^p K_i$.  For each $i \in \{1,2, \dots p\}$, observe that $K_i \cap P$ consists of \emph{consecutive} \emph{green points} or \emph{consecutive} \emph{black points}.  We \emph{partition} $f(K_i \cap P)$ into \emph{maximal blocks} $B_{ij}$, $j \in \{1,2,  $ $ \dots n_i\}$ of \emph{consecutive} \emph{i-green points} and \emph{consecutive} \emph{i-black points}. It follows that the \emph{union} of the $B_{ij}$s, $ i=1,2, \dots p$ and $j=1,2,\dots n_i$ \emph{yields} $P$.

 Let $Z_{ij} = f^{-1}(B_{ij})$ for all $ i=1,2, \dots p$ and $j=1,2,\dots n_i$. Then, from Theorem \ref{nec:green}, $Z_{ij}s$ represent \emph{set} of \emph{consecutive} \emph{green points} or \emph{set} of \emph{consecutive} \emph{black points} and they are mapped onto $B_{ij}$s \emph{homeomorphically}. Hence, the \emph{union} of $Z_{ij}$s, $ i=1,2, \dots p$ and $j=1,2,\dots n_i$ also yields the \emph{set} $P$,  $f|_{Z_{ij}}$ is \emph{monotone} and maps without any \emph{expansion}. Thus, from Theorem \ref{P:linear:IET}  $f|_P$ is \emph{conjugate} to an IET with \emph{intervals of isometry}:  $Z_{ij},   i=1,2, \dots p, j=1,2,\dots n_i $, \emph{restricted} on one of its \emph{cycles}. 
	
	It remains to show that the number of such \emph{sets} $Z_{ij}$, $ i=1,2, \dots p$ and $j=1,2,\dots n_i$ is \emph{bounded} by some \emph{function} of $m$. By Theorem \ref{second},  $ p \leqslant m+2$. Also, observe that each such \emph{block} $B_{ij}$, is \emph{contained either} in a \emph{distinct} \emph{i-green} \emph{island} or an \emph{i-black island}. Thus, $n_i$ is the \emph{precisely} equal to \emph{sum} of the \emph{number} of \emph{i-green} and \emph{i-black} \emph{island} which \emph{intersects} $f(K_i \cap P)$. So, by  Theorem \ref{fifth}, $n_i <  m+1$. Hence, the \emph{result} follows. 
\end{proof}

 Description of the \emph{dynamics} of all possible  \emph{unimodal} and \emph{bimodal over-twist patterns} were obtained in \cite{BS} and \cite{BB2} respectively.  Using these we can obtain the \emph{exact} number of \emph{segments} of \emph{isometry} in the \emph{unimodal} and \emph{bimodal} case. We refer to the following result from \cite{BB2}: 

\begin{theorem}[\cite{BB2}]\label{des:bim}
	Given any rational over-rotation  number $\frac{p}{q}$, where $p$ and $q$ are co-prime,  there are $q-2p+1$ distinct over-twist patterns:  $\Gamma_{r, \frac{p}{q}}$, $r = 0, 1, 2, \dots, q-2p $ of over-rotation number $\frac{p}{q}$ and modality less than or equal to $2$. Let $P_{r, \frac{p}{q}}$ be a cycle,  with its elements spatially labeled as $x_1, x_2, \dots, x_q$, which exhibits $\Gamma_{r, \frac{p}{q}}$  and let $f$ be a $P_{r, \frac{p}{q}}$-linear map. Then, $f|_{P_{r, \frac{p}{q}}}$ acts as follows: 

\begin{enumerate}
	\item  The first $r$ points, $x_1, x_2, \dots x_r$ of $P_{r, \frac{p}{q}}$ are shifted to the right by $p$ points, that is, $f(x_i) = x_{i+p}$ for $i=1,2, \dots r$. 
	
	\item 	The next $p$ points of $P_{r, \frac{p}{q}}$: $x_{r+1}, x_{r+2}, \dots x_{r+p}$
	map onto the last $p$ points  of $P_{r, \frac{p}{q}}$  with a flip, that is, the  orientation is 
	reversed but without any expansion, that is, $f(x_{r+i}) = x_{q-i+1}$ for $i=1,2, \dots p$. 
	
	\item The next $p$ points of $P_{r, \frac{p}{q}}$: $x_{r+p+1}, x_{r+p+2}, \dots x_{r+2p}$
	map onto the first $p$ points of $P_{r, \frac{p}{q}}$   with a flip, that is, the  orientation is 
	reversed but without any expansion, that is, $f(x_{r+p+i}) = x_{p-i+1}$ for $i=1,2, \dots p$. 
	
	\item The last $s=q-2p-r$ points of $P_{r, \frac{p}{q}}$: $x_{r+2p+1}, x_{r+2p+2}, \dots x_{q}$are shifted to the left by $p$ points, that is, $f(x_{r+2p+i}) = x_{r+p+i}$ for $i=1,2, \dots q-2p-r$. 
\end{enumerate}

	\end{theorem}

\begin{remark}\label{r:0:case}
		Out of these $q-2p+1$ distinct \emph{over-twist patterns}:  $\Gamma_{r, \frac{p}{q}}$, $r = 0, 1, 2, \dots, q-2p $,  the \emph{patterns} corresponding to $r=0$ and $r=q-2p$ are \emph{unimodal} and the remaining $q-2p-1$ patterns are \emph{strictly bimodal}.
\end{remark}

\begin{theorem}\label{bimodal:IET:1}
Let $P$ be a cycle of $P$-linear map $f$ which exhibits an over-twist pattern with modality $m =1,2$ and let $f:[0,1] \to [0,1]$ be a $P$-linear map. Then, $f|_P$ is conjugate to an $(4,2) -\IET$ if $m=2$ and to a $(3,2) -\IET$ if $m=1$, restricted on one of its cycle. 
\end{theorem}

\begin{proof}
Suppose $P$  has \emph{over-rotation number} $\frac{p}{q}$ and its elements be spatially labeled as $x_1, x_2, \dots, x_q$. Choose $r \in \{ 0,1,2, $ $  \dots q-2p\}$ such that $P$ \emph{exhibits} the \emph{over-twist pattern}  $\Gamma_{r, \frac{p}{q}}$  with \emph{over-rotation number} $\frac{p}{q}$. We \emph{partition} $P$ into $4$ \emph{disjoint parts}: $P_1 = \{ x_1, x_2, \dots x_r\}$, $P_2 = \{x_{r+1}, x_{r+2}, \dots, x_{r+p}\}$,  $P_3 = \{x_{r+p+1}, $ $ x_{r+p+2},$ $ \dots, x_{r+2p}\}$, $P_4 = \{x_{r+2p+1}, $ $ x_{r+2p+2},$ $ \dots, x_{q}\}$. From Theorem \ref{des:bim}, $f|_{K_i}$, is \emph{monotone} for all \textbf{$i \in \{1,2,3,4\}$ } and \emph{maps} \emph{without any expansion}. The \emph{orientation} is \emph{preserved} in case of $P_1$ and $P_4$ and \emph{reversed} in case of $P_2$ and $P_3$. Also, from Remark \ref{r:0:case}, if $\card(P_1) =0$ or $\card(P_4) = 0$, then $P$ becomes \emph{unimodal}, otherwise $P$ becomes \emph{strictly bimodal}. Hence, by Theorem \ref{P:linear:IET}, the result follows.

\end{proof}

\begin{figure}[H]
	\caption{\emph{Partition} of the \emph{cycle} $P$ \emph{exhibiting} \emph{pattern} $\Gamma_{3, \frac{3}{11}}$ into $4$ sets: $P_1, P_2, P_3$ and $P_4$}
	\centering
	\includegraphics[width=1 \textwidth]{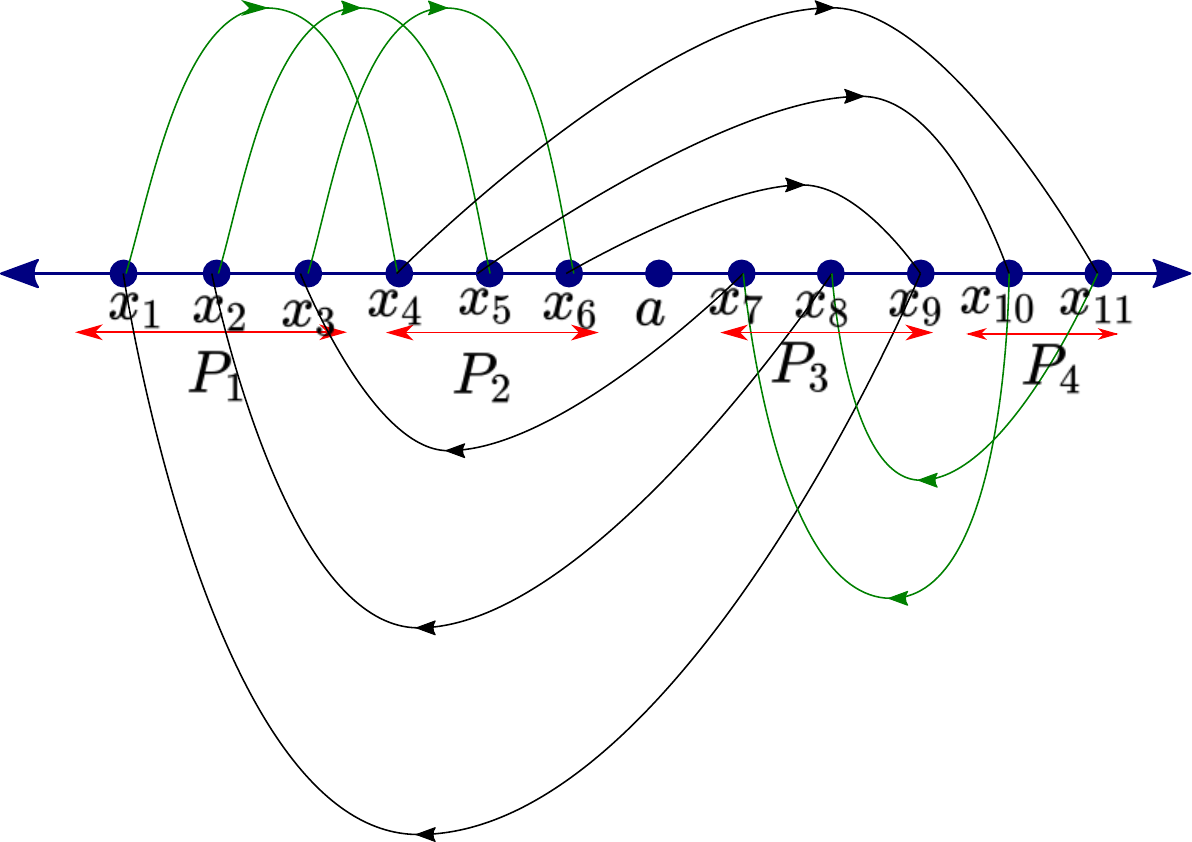}
	\label{drawing5}
\end{figure}


\begin{thebibliography}{99999999}
	
	
	
	%\bibitem{alm98} L. Alsedà, J. Llibre, M. Misiurewicz, \emph{Periodic orbits of maps of Y} , Trans. Amer. Math. Soc. \textbf{313} (1989) 475–538.




%\bibitem{akm65}	R. Adler, A. Konheim, and M. McAndrew,
%\emph{Topological entropy}, Trans. Amer. Math. Soc. \textbf{114} (1965),
%309--319.


\bibitem{alm00} Ll. Alsed\`{a}, J. Llibre and M. Misiurewicz,
\emph{Combinatorial Dynamics and Entropy in Dimension One},
Advanced Series in Nonlinear Dynamics (2nd edition) \textbf{5} (2000),
World Scientific Singapore (2000)

%\bibitem{almnew}  Ll. Alsed\`{a}, J. Moreno, \emph{Linear orderings and the full periodicity kernel for the $n$-star}, Math.Anal. Appl. \textbf{180} (1993) 599-616. 

%\bibitem{almm88} Ll. Alsed\`a, J. Llibre, F. Ma\~nosas and M.
%Misiurewicz, \emph{Lower bounds of the topological entropy for
%continuous maps of the circle of degree one}, Nonlinearity
%\textbf{1}(1988), 463--479.

%\bibitem{ak79} J. Auslander, Y. Katznelson, \emph{Continuous maps
%	of the circle without periodic points,} Israel J. of Math.
%\textbf{32} (1979), 375--381.

\bibitem{Ba} S. Baldwin, \emph{Generalization of a theorem of
	Sharkovsky on orbits of continuous real valued functions,} Discrete
Math. \textbf{67} (1987), 111--127.



%\bibitem{BB3} S. Bhattacharya, A. Blokh, \emph{Monotonicity of over-rotation intervals for bimodal interval maps}, Topology and its Applications \textbf{308} (2022),108004.  


\bibitem{BB2} S. Bhattacharya, A. Blokh, \emph{Over-rotation intervals of bimodal interval maps},
Journal of Difference Equations and Applications \textbf{26} (2020), 1085-1113. 


%\bibitem{BB1} S. Bhattacharya, A. Blokh, \emph{Very badly ordered
%	cycles of interval maps}, Journal of Difference Equations and
%Applications \textbf{26} (2020), 1067-1084. 

%\bibitem{BGKT} Boschernitzan.M, Galperin.G, Kruger.T, Troubetzkoy.S, \emph{Periodic billiard orbits are dense in rational polygons}, Trans. Amer. Math. Soc. \textbf{350} (1998), 3523-3535. 




%\bibitem{bgmy80} L. Block, J. Guckenheimer, M. Misiurewicz and
%L.-S. Young, \emph{Periodic points and topological entropy of
%	one-dimensional maps}, Springer Lecture Notes in Mathematics
%\textbf{819} (1980), 18--34.

\bibitem{BKM} Boldrighini.C, Keane.M, Marchetti.F, \emph{Billiards in polygons}, Ann. Probab. \textbf{6} (1978), 532-540. 



\bibitem{BS} A. Blokh, K. Snider, \emph{Over-rotation numbers for
	unimodal maps,} Journal of Difference Equations and Aplications
\textbf{19}(2013), 1108--1132.




\bibitem{blo95a} A. Blokh, \emph{The Spectral Decomposition for
	One-Dimensional Maps}, Dynamics Reported \textbf{4} (1995), 1--59.


\bibitem{B1} A. Blokh, \emph{On Rotation Intervals for Interval
	Maps}, Nonlinearity \textbf{7}(1994), 1395--1417.



\bibitem{B2} A. Blokh, \emph{Rotation Numbers, Twists and a
	Sharkovsky-Misiurewicz-type Ordering for Patterns on the Interval},
Ergodic Theory and Dynamical Systems \textbf{15}(1995), 1--14. 	



%\bibitem {B3} A. Blokh,  \emph{Functional Rotation Numbers for
%	One-Dimensional Maps,} Trans. Amer. Math. Soc. \textbf{347}(1995),
%499--514

\bibitem{BM1} A. Blokh, M. Misiurewicz, \emph{A new order for
	periodic orbits of interval maps}, Ergodic Theory and Dynamical Sys.
\textbf{17}(1997), 565-574

%\bibitem{BMR} A. Blokh, M. Misiurewicz, \emph{Rotation numbers for certain maps of an n-od}, Topology and its Applications
%\textbf{114} (2001) 27–48

	%\bibitem{BME} A. Blokh, M. Misiurewicz, \emph{Evolution of the Sharkovsky Theorem}, Ukrains'kyi Mathematychnyi Zhurnal \textbf{76} (2024), 48-61

\bibitem{BM2} A. Blokh, M. Misiurewicz, \emph{Rotating an interval
	and a circle}, Trans. Amer. Math. Soc. \textbf{351}(1999), 63--78.

%\bibitem{BM3} A. Blokh, M. Misiurewicz, \emph{Entropy and Over-Rotation Numbers for Interval Maps}, 
%Mathematical Proceedings of Steklov Institute, dedicated to D. V. Anosov's 60th birthday \textbf{216}(1997), 229-235. 

\bibitem{BS} A. Blokh, K. Snider, \emph{Over-rotation numbers for
	unimodal maps,} Journal of Difference Equations and Applications
\textbf{19}(2013), 1108--1132.

%	\bibitem{BOM} A. Blokh, O.M Sharkovsky, \emph{Sharkovsky Ordering}, Springer Briefs in Mathematics, (2022).


%\bibitem{Bo} J. Bobok, \emph{Twist systems on the interval}, Fund.
%Math. \textbf{175}(2002), 97--117.

%\bibitem{bk98} J. Bobok and M. Kuchta \emph{X-minimal orbits for
%	maps on the interval}, Fund. Math. \textbf{156}(1998), 33--66.

%\bibitem{cgt84} A. Chenciner, J.-M. Gambaudo and C. Tresser
%\emph{Une remarque sur la structure des endo\-morphismes de degr\'e
%	$1$ du cercle}, C. R. Acad. Sci. Paris, S\'er I Math. \textbf{299}
%(1984), 145--148.

%\bibitem{dgs76} M. Denker, C. Grillenberger,
%K. Sigmund, \emph{Ergodic theory on compact spaces,} Lecture Notes
%in Mathematics \textbf{527}(1976) Springer-Verlag, Berlin-New
%York.

\bibitem{EPS} Esp'in.J.G, Peralta.D, Soler.G, \emph{Existence of minimal flows on non orientable surfaces}, Discrete and Continuous Dynamical Systems, \textbf{37}(2017), 4191-4211

%\bibitem{ito81} R. Ito, \emph{Rotation sets are closed}, Math.
%Proc. Camb. Phil. Soc. \textbf{89}(1981), 107--111.

\bibitem{KM} Keane.M, \emph{Interval exchange transformations}, Math.Z. \textbf{141}(1975), 25-31.



%\bibitem{MT} J. Milnor and W. Thurston, \emph{On Iterated Maps on
%	the Interval}, Lecture Notes in Mathematics, Springer, Berlin
%\textbf{1342}(1988), 465--520.

%\bibitem {mis82} M. Misiurewicz, \emph{Periodic points of maps
%	of degree one of a circle,} Ergod. Th. \& Dynam. Sys. \textbf{2}(1982)
%221--227.

%\bibitem {mis89}M. Misiurewicz, \emph{Formalism for studying
%	periodic orbits of one dimensional maps}, European Conference
%on Iteration Theory (ECIT 87), World Scientific
%Singapore (1989), 1--7.	

%\bibitem{MN} M. Misiurewicz and Z. Nitecki,  \emph{Combinatorial Patterns for
%	maps of the interval}, Mem. Amer. Math. Soc. \textbf{456}(1990) 		

%\bibitem{mz89} M. Misiurewicz and K. Ziemian, \emph{Rotation Sets
%	for Maps of Tori}, J. Lond. Math. Soc. (2) \textbf{40}(1989), 490--506.

\bibitem{NA} Nogueira.A, \emph{Almost all interval exchange transformations with flips are nonergodic}, Ergod. Theory Dynam. Syst. \textbf{9}(1989), 515-525

%\bibitem{npt83} S. Newhouse, J. Palis, F. Takens
%\emph{Bifurcations and stability of families of diffeomorphisms},
%Inst. Hautes \'Etudes Sci. Publ. Math. \textbf{57}(1983), 5--71.

%\bibitem{poi} H. Poincar\'e, \emph{Sur les courbes d\'efinies par
%	les \'equations diff\'erentielles}, Oeuvres completes, \textbf{1}
%137--158, Gauthier-Villars, Paris (1952).

%\bibitem{rt86} F. Rhodes, C. Thompson, \emph{Rotation numbers for
%	monotone functions on the circle}, J. London Math. Soc. \textbf{34}(1986), 360--368.

\bibitem{shatr} A. N. Sharkovsky, \emph{Coexistence of the
	cycles of a continuous mapping of the line into itself}, Internat.
J. Bifur. Chaos Appl. Sci. Engrg. \textbf{5}(1995), 1263--1273.

\bibitem{VM} Viana. M, \emph{Ergodic theory of interval exchange maps}, Rev. Mat. Complut \textbf{19}(2006), 7-100

%\bibitem {zie95}  K. Ziemian, \emph{Rotation sets for subshifts
%	of finite type,} Fundam. Math. \textbf{146}(1995), 189--201.




\end{thebibliography}
\end{document}